\documentclass[12pt]{amsart}

\usepackage{amscd, amsmath, amssymb, amsthm, amsfonts, amsxtra, 
enumerate, latexsym, tabularx, setspace}
\usepackage[all]{xy}
\usepackage[dvips]{graphicx,color}

\def\FF{{\mathbb{F}}}

\def\PP{{\mathbb{P}}}
\def\Q{{\mathbb{Q}}}
\def\R{{\mathbb{R}}}

\def\ZZ{{\mathbb{Z}}}

\def\ch{{\mathrm{char\;}}}

\def\Der{{\mathrm{Der}}}

\def\Frac{{\mathrm{Frac}}}

\def\Hom{{\mathrm{Hom}}}
\def\hom{{\mathcal{H}om}}

\def\Mod{{\mathrm{mod\; }}}
\def\ord{{\mathrm{ord}}}

\def\Proj{{\mathrm{Proj\; }}}

\def\Sing{{\mathrm{Sing}}}

\def\Spec{{\mathrm{Spec}}}

\def\der{\partial}

\theoremstyle{plain}
\newtheorem{thm}{Theorem}[section]
\newtheorem*{theorem}{Theorem}

\newtheorem{lmm}[thm]{Lemma}

\theoremstyle{definition}

\newtheorem{exa}[thm]{Example}

\theoremstyle{remark}
\newtheorem{rem}{Remark}[section]
\newtheorem*{prob}{Problem}

\title{Classification of globally F-regular $F$-sandwiches of Hirzebruch surfaces}
\author{Tadakazu Sawada}
\address{Department of General education, Fukushima National College of Technology, 
30 Aza-Nagao, Kamiarakawa, Iwaki-shi, Fukushima 970-8034, Japan}
\email{sawada@fukushima-nct.ac.jp}

\begin{document}

\begin{abstract}
Let $X$ be a smooth variety over an algebraically closed field of positive characteristic. An 
$F$-sandwich of $X$ is a normal variety $Y$ through which the relative Frobenius morphism 
of $X$ factors as $F:X\rightarrow Y \rightarrow X$. In this paper, we give a classification of 
globally F-regular $F$-sandwiches of Hirzebruch surfaces.
\end{abstract}

\maketitle
\markboth{Tadakazu Sawada}{Classification of globally F-regular 
$F$-sandwiches of Hirzebruch surfaces}
\section*{Introduction}
Let $X$ be a smooth variety over an algebraically closed field of positive characteristic. 
A Frobenius sandwich of $X$ is a normal variety $Y$ through which the (iterated) relative 
Frobenius morphism of $X$ factors as $F:X\rightarrow Y \rightarrow X$. For a given variety 
$X$, it is natural to ask what kinds of singularities and varieties appear as Frobenius sandwiches 
of $X$. However, it seems hopeless to classify the Frobenius sandwiches explicitly without any 
restriction to those under consideration because of pathological phenomena in positive 
characteristic. For example, there exists a Frobenius sandwich of the projective plane $\PP^2$ 
whose nonsingular model is a surface of general type. (In characteristic $0$, every unirational 
surface is rational.) Taking into account such pathological phenomena, we consider Frobenius 
sandwiches that behave better in the sense of Frobenius splitting, that is, globally F-regular 
Frobenius sandwiches. Global F-regularity is defined via splitting of Frobenius morphisms and has 
remarkably nice properties. Assuming global F-regularity, we expect to exclude pathological cases, 
so that we may hope for comprehensive study of Frobenius sandwiches. 

We consider the following problem: 
\begin{prob}
Given a globally F-regular variety $X$, classify globally F-regular Frobenius sandwiches of $X$.
\end{prob}

We dealt with the simplest case where $X=\PP^2$ in \cite{HS}. We showed that globally F-regular 
$F$-sandwiches of $\PP^2$ of degree $p$ (see Section 1 for the definition of degree) are singular 
toric surfaces and there are $p-1$ isomorphism classes. In this paper, we give a classification of 
globally F-regular $F$-sandwiches of Hirzebruch surfaces of degree $p$. In particular, we see that 
those globally F-regular $F$-sandwiches are toric surfaces and there are $p$ or $p+1$ isomorphism 
classes for each Hirzebruch surface. The following is the main result: 

\begin{theorem}
A globally $F$-regular $F$-sandwich of the Hirzebruch surface $\mathcal{H}_d$ of degree $p$ is 
isomorphic to either one of the toric surfaces $X_{\Sigma_{di}}\, (0 \leq i \leq p)$. 
\end{theorem}

\noindent Here $X_{\Sigma_{di}}$ stands for the toric surface associated to a fan $\Sigma_{di}$. The 
fans $\Sigma_{di}$ are given by considering the original fan on $N=\ZZ^2$ associated to $\mathcal{H}_d$ 
on an overlattice $N+\ZZ \frac{1}{p}(1, i)\ (0\leq i\leq p)$. Hence we easily see that $X_{\Sigma_{di}}$ 
are globally F-regular $F$-sandwiches of $\mathcal{H}_d$ of degree $p$. Conversely, this theorem says 
that globally F-regular $F$-sandwiches of $\mathcal{H}_d$ of degree $p$ are only those natural ones. 

In Section 1, we review generalities on Frobenius sandwiches and globally F-regular varieties. In Section 2, 
we give the classification. 

\section{Preliminary}
We work over an algebraically closed field $k$ of characteristic $p>0$. Let $X$ be an algebraic variety 
over $k$. The absolute Frobenius morphism $F: X \rightarrow X$ is the identity on the underlying 
topological space of $X$, and the $p$-th power map on the structure sheaf $\mathcal{O}_X$, which we also 
denote by $F: \mathcal{O}_X \rightarrow F_{\ast}\mathcal{O}_X$. Let $X^{(-1)}$ be the base change of $X$ by the absolute 
Frobenius morphism of $\Spec\, k$. The relative Frobenius morphism $F_{\rm rel}: X \rightarrow X^{(-1)}$ 
is defined by the following Cartesian square: 
$$\xymatrix
{X \ar@/^8mm/[rr]^F \ar[r]^{\!\!\!\! F_{\rm rel}} \ar[dr] & X^{(-1)} \ar[r] \ar[d] & X \ar[d] \\
                                                          & \Spec\, k \ar[r]^F     & \Spec\, k
}$$

In what follows, we use these variants of Frobenius morphisms interchangeably. Since we work over the 
algebraically closed field, we need not strictly distinguish these variants. 

{\bf Frobenius sandwiches.} First we review generalities on Frobenius sandwiches. Let $X$ be a smooth 
variety over $k$. A normal variety $Y$ is an {\it $F^e$-sandwich} of $X$ if the $e$-th iterated relative 
Frobenius morphism of $X$ factors as 
$$\xymatrix
{X \ar[rr]^{F^e_{\rm rel}} \ar[rd]_{\pi}  &                  & X^{(-e)} \\
                                          & Y \ar[ru]_{\rho} &
}$$
for some finite $k$-morphisms $\pi : X \rightarrow Y$ and $\rho : Y \rightarrow X^{(-e)}$, which are 
homeomorphisms in the Zariski topology. An {\it $F$-sandwich} will mean an $F^1$-sandwich. We say that 
the Frobenius sandwich $Y$ is of {\it degree $p$} if the degree of the morphism $\pi : X \rightarrow Y$ is $p$. 

By a {\it $1$-foliation} of $X$, we mean a saturated $p$-closed subsheaf $L$ of the tangent bundle $T_X$ 
closed under Lie brackets, where $L$ is said to be $p$-closed if it is closed under $p$-times iterated 
composite of differential operators; see \cite{E}. 

It is known that there are one-to-one correspondences among the followings (see \cite{RS}, \cite{E}, \cite{Hirokado}):
\begin{itemize}
\item $F$-sandwiches of $X$ of degree $p$; 
\item invertible $1$-foliations of $X$; 
\item $p$-closed rational vector fields of $X$ modulo an equivalence $\sim$.
\end{itemize}

\noindent (1) $F$-sandwiches and $1$-foliations: The correspondence is given by
$$Y \mapsto L=\{\delta\in T_X\,|\,\delta(f)=0\text{ for all }f\in \mathcal{O}_Y\}\subset T_X$$
and
$$L \mapsto 
\mathcal{O}_Y=\{f\in\mathcal{O}_X\,|\,\delta(f)=0\text{ for all }\delta\in L\}\subset \mathcal{O}_X.$$
(The inclusion $\mathcal{O}_Y\subset \mathcal{O}_X$ induces the finite morphism $\pi:X\rightarrow Y$.) The well-definedness 
of this correspondence is guaranteed by the Galois correspondence due to Aramova and Avramov \cite{AA}. 

\noindent (2) $F$-sandwiches and rational vector fields: We define an equivalence relation $\sim$ between 
rational vector fields $\delta,\delta' \in \Der_k\, k(X)$ as follows: $\delta \sim {\delta}'$ if and only if there 
exists a non-zero rational function $\alpha \in k(X)$ such that $\delta = \alpha {\delta}'$. Let 
$\{U_i = \Spec\, R_i\}_i$ be an affine open covering of $X$. Given a $p$-closed rational vector field 
$\delta \in \Der_k\, k(X)$, we have a quotient variety $X/{\delta}$ defined by glueing $\Spec\, R_i^{\delta}$, 
where $R_i^{\delta}=\{r\in R\,|\,\delta(r)=0\}$, and a quotient map $\pi_{\delta} : X \rightarrow X/{\delta}$ 
induced from the inclusions $R_i^{\delta} \subset R_i$. Then we easily see that $R^{\delta}_i$ is normal and the 
field extension ${\Frac\, R_i}/{\Frac\, R_i^{\delta}}$ is purely inseparable of degree $p$. This means that 
$X/\delta$ is an $F$-sandwich of degree $p$ with the finite morphism $\pi_{\delta} : X \rightarrow X/{\delta}$ 
through which the Frobenius morphism of $X$ factors. Conversely, if $Y$ is an $F$-sandwich of $X$ of degree 
$p$ with the finite morphism $\pi : X \rightarrow Y$ through which the Frobenius morphism of $X$ factors, 
then there exists a rational vector field $\delta$ such that $\pi = \pi_{\delta}$ and $Y = X/{\delta}$. Indeed, 
there exists a $p$-closed rational vector field $\delta \in \Der_k\, k(X)$ such that $k(X)^{\delta}=k(Y)$ by 
Baer's result (see e.g., \cite{J}), since the field extension $k(X)/k(Y)$ is purely inseparable of degree $p$. Thus 
$\delta$ induces an inclusion $\mathcal{O}_Y \subset \mathcal{O}_{X/{\delta}}$, so that there exists a finite birational morphism 
$X/{\delta} \rightarrow Y$. Since $Y$ is normal, this morphism is an isomorphism.

\noindent (3) $1$-foliations and rational vector fields: A rational vector field $\delta \in \Der_k\, k(X)$ is locally 
expressed as $\alpha \sum f_i \der/\der s_i$, where $s_i$ are local coordinates, $f_i$ are regular functions 
without common factors and $\alpha \in k(X)$. The divisor $\mathrm{div}(\delta)$ associated to $\delta$ is 
defined by glueing the divisors $\mathrm{div}(\alpha)$ on affine open sets. Then $\mathcal{O}_X(\mathrm{div}(\delta))$ 
has a saturated $p$-closed invertible subsheaf structure of $T_X$: 
$$\begin{array}{cccc}
0 \longrightarrow & \mathcal{O}_X(\mathrm{div}(\delta)) & \xrightarrow{\cdot \delta} & T_X \vspace{2mm} \\
                  & h/{\alpha}       & \longmapsto                & h \sum f_i \der/\der s_i
\end{array}$$
where $h$ is a regular function. Then we see that $\delta \mapsto \mathcal{O}_X(\mathrm{div}(\delta))$ gives a 
one-to-one correspondence between $p$-closed rational vector fields modulo the equivalence and invertible 
$1$-foliations.

\medskip Let $X$ be a smooth projective surface over $k$ and $L$ an invertible $1$-foliation of $X$. We have 
an exact sequence
$$0 \longrightarrow L \longrightarrow T_X \longrightarrow I_Z \otimes L' \longrightarrow 0,$$
where $I_Z$ is the defining ideal sheaf of a zero-dimensional subscheme $Z$ and $L'$ is an invertible sheaf. We 
call the support of $Z$ the {\it singular locus} of $L$ and denote it by $\Sing\, L$. Let $Y$ be the Frobenius 
sandwich of degree $p$ of $X$ corresponding to $L$. 
Then $\Sing\, Y = \pi (\Sing\, L)$, where $\Sing\, Y$ is the singular locus of $Y$, and outside $\Sing\, L$ 
we have the canonical bundle formula $\omega_X \cong \pi^*\omega_Y \otimes L^{\otimes {(p-1)}}$; see \cite{E}. 

Let $\delta \in \Der_k\, k(X)$, and assume that $\delta$ is locally expressed as 
$\alpha ( f {\partial}/{\partial s} + g {\partial}/{\partial t})$ where $s, t$ are local coordinates, $f, g$ are regular 
functions without common factors and $\alpha \in k(X)$. Then $\Sing\, \mathcal{O}_X (\mathrm{div}(\delta))$ is defined 
locally by $f=g=0$. 

{\bf Globally F-regular varieties.} Next we review generalities on globally F-regular varieties. See \cite{Smith}, 
\cite{SS} for further details. Let $X$ be a projective variety over $k$. We say that $X$ is {\it globally 
$F$-regular} if for any effective Cartier divisor $D$ on $X$, there exists a power $q=p^e$ such that the 
composition map 
$\mathcal{O}_X \xrightarrow{F^e} F^e_{\ast} \mathcal{O}_X \xrightarrow{F^e_{\ast}(s)\, \cdot} F^e_{\ast} \mathcal{O}_X (D)$ 
splits as an $\mathcal{O}_X$-module homomorphism, where $s$ is a section of $\mathcal{O}_X (D)$ vanishing precisely along $D$. 

\begin{exa}
\begin{enumerate}[(1)]
\item Any projective toric variety is globally F-regular; see \cite{Smith}. 
\item Let $\ch k = p > 5$. Then any smooth del Pezzo surface is globally F-regular; see \cite{Hara}. 
\end{enumerate}
\end{exa}

Let $X$ be a $\Q$-Gorenstein globally F-regular variety and $H$ be an ample effective divisor. By the definition, 
there exists $e \geq 1$ such that the map $\mathcal{O}_X \rightarrow F^e_{\ast} \mathcal{O}_X (H)$ splits as an 
$\mathcal{O}_X$-module homomorphism. On the other hand, we have 
$\hom_{\mathcal{O}_X} (F^e_{\ast} \mathcal{O}_X (H), \mathcal{O}_X) \cong F^e_{\ast} \mathcal{O}_X ((1-p^e)K_X -H)$ by the adjunction formula. 
Thus a splitting $F^e_{\ast} \mathcal{O}_X (H) \rightarrow \mathcal{O}_X$ is its non-zero global section, so that there exists an 
effective divisor $D \sim (1-p^e)K_X -H$. This means that $-K_X$ is big. More generally, Schwede and Smith 
showed in \cite{SS} the following: If $X$ is a globally F-regular variety, then there exists an effective $\Q$-divisor 
$\Delta$ on $X$ such that the pair $(X, \Delta)$ is log Fano. This strong restriction on the structure of varieties is 
a motivation for  the problem raised at the beginning. 

The following lemma is a global variant of the well-known fact that a pure subring of a strongly F-regular ring is 
strongly F-regular. We include the proof for the reader's convenience; see \cite{SS} for a general case. 

\begin{lmm}\label{split_vs_F-reg}
Let $X$ be a globally F-regular variety over $k$ and $Y$ be an $F^e$-sandwich of $X$ with the finite morphism 
$\pi : X \rightarrow Y$ through which the Frobenius morphism of $X$ factors. Then $Y$ is globally F-regular if and 
only if the associated ring homomorphism $\mathcal{O}_Y \rightarrow \pi_{\ast} \mathcal{O}_X$ splits as an 
$\mathcal{O}_Y$-module homomorphism. 
\end{lmm}
\begin{proof}
Suppose that $Y$ is globally F-regular. Then $Y$ is F-split, i.e., the Frobenius ring homomorphism 
$\mathcal{O}_Y \rightarrow F_{\ast} \mathcal{O}_Y$ splits as an $\mathcal{O}_Y$-module homomorphism. Hence the 
Frobenius map $\mathcal{O}_Y \rightarrow F^e_{\ast} \mathcal{O}_Y$ splits as an $\mathcal{O}_Y$-module homomorphism. Now this map 
factors as $\mathcal{O}_Y \rightarrow \pi_{\ast} \mathcal{O}_X \rightarrow F^e_{\ast} \mathcal{O}_Y$, since $Y$ is an $F^e$-sandwich of $X$. 
Thus the ring homomorphism $\mathcal{O}_Y \rightarrow \pi_{\ast} \mathcal{O}_X$ splits as an $\mathcal{O}_Y$-module homomorphism.

Next suppose that the ring homomorphism $\mathcal{O}_Y \rightarrow \pi_{\ast} \mathcal{O}_X$ splits. Let $D$ be an 
effective Cartier divisor. Since $X$ is globally F-regular, there exists $e \geq 1$ such that the map 
$\mathcal{O}_X \rightarrow F^e_{\ast}\mathcal{O}_X(\pi^{\ast} D)$ splits as an $\mathcal{O}_X$-module homomorphism. Then we see that the 
composition map $\mathcal{O}_Y \rightarrow F^e_{\ast}\mathcal{O}_Y(D) \rightarrow \pi_{\ast} F^e_{\ast}\mathcal{O}_X(\pi^{\ast} D)$ splits as an 
$\mathcal{O}_Y$-module homomorphism from the commutative diagram: 
$$\xymatrix{
 \mathcal{O}_Y \ar[r] \ar[d]         & \pi_{\ast} \mathcal{O}_X \ar[d] \\
 F^e_{\ast} \mathcal{O}_Y (D) \ar[r] & \pi_{\ast} F^e_{\ast} \mathcal{O}_X (\pi^{\ast}D)
}$$
Therefore $\mathcal{O}_Y \rightarrow F^e_{\ast}\mathcal{O}_Y(D)$ splits as an $\mathcal{O}_Y$-module homomorphism. 
\end{proof}

We will use the following lemma in the proof of the main theorem. 

\begin{lmm}[\cite{HS} Lemma 3.2]\label{non-nil}
Let $S=k[x,y]$ and let $\delta=f{\partial}/{\partial x}+g{\partial}/{\partial x} \in \Der_k\,S$, where $f, g \in (x,y)$ and 
have no common factors. Suppose $\delta$ is $p$-closed. If the inclusion map $S^{\delta} \hookrightarrow S$ splits 
as an $S^{\delta}$-module, then $\delta$ is not nilpotent. 
\end{lmm}

\section{Classification of globally F-regular $F$-sandwiches of Hirzebruch surfaces}
First we consider $F$-sandwiches of the projective plane $\PP^2$. Let $X_0$, $X_1$ and $X_2$ be homogeneous 
coordinates of $\PP^2$, i.e., $\PP^2 = \Proj k [X_0, X_1, X_2]$. Let $x=X_1/X_0$, $y=X_2/X_0$ (resp. $z=X_0/X_1$, 
$w=X_2/X_1$ ; $u=X_0/X_2$, $v=X_1/X_2$) be the affine coordinates of $U_0 := D_+ (X_0)$ (resp. $U_1 := D_+ (X_1)$ 
; $U_2 := D_+ (X_2)$). 

\begin{exa}
We give examples of $F$-sandwiches of $\PP^2$ of degree $p$. 
\begin{enumerate}[\normalfont \rmfamily (1)]
\item Let $\delta = x\der/\der x + y \der/\der y \in \Der_k\, k(\PP^2)$ and $\pi: \PP^2 \rightarrow Y=\PP^2/\delta$ 
the quotient map. If we express $\delta$ for the local coordinates $z, w$ and $u, v$, we have 
$$\delta = -z\der/\der z = -u\der/\der u.$$
Then the corresponding $1$-foliation $\mathcal{O}_{\PP^2} (\mathrm{div}(\delta))$ is isomorphic to $\mathcal{O}_{\PP^2} (1)$ and 
\begin{eqnarray*}
\mathcal{O}_{\pi(U_0)} &=& \mathcal{O}_{U_0}^{\delta} = k[x^p, x^{p-1}y, \ldots,y^p], \\%[1mm]
\mathcal{O}_{\pi(U_1)} &=& \mathcal{O}_{U_1}^{\delta} = k[z^p, w], \\%[1mm]
\mathcal{O}_{\pi(U_2)} &=& \mathcal{O}_{U_2}^{\delta} = k[u^p, v].
\end{eqnarray*}
Hence $Y$ has a toric singularity of type $\frac{1}{p}(1,1)$ on $\pi(U_0)$. (We say a singularity $(X,x)$ is a toric singularity 
of type $\frac{1}{p}(1,n)$, if the completion of $\mathcal{O}_{X,x}$ is isomorphic to 
$k[[\,x^i y^j\,|\,i+nj\equiv 0\ \mathrm{mod}\,p\,]]$.) We identify $Y$ with $\PP^2$ as a topological space via the 
homeomorphism $\pi : \PP^2 \rightarrow Y$ on the underlying topological spaces. Then the configuration of the singular 
points of $Y$ is as follows:\vspace{4mm}\\
\center{
%WinTpicVersion4.08
\unitlength 0.1in
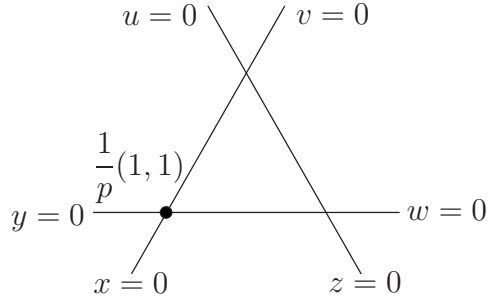
\begin{picture}( 16.0000, 14.0000)(  4.0000,-16.0000)
% LINE 2 0 3 0 Black White
% 2 600 1600 1400 200
% 
{\color[named]{Black}{%
\special{pn 8}%
\special{pa 600 1600}%
\special{pa 1400 200}%
\special{fp}%
}}%
% LINE 2 0 3 0 Black White
% 2 1000 200 1800 1600
% 
{\color[named]{Black}{%
\special{pn 8}%
\special{pa 1000 200}%
\special{pa 1800 1600}%
\special{fp}%
}}%
% LINE 2 0 3 0 Black White
% 2 400 1280 2000 1280
% 
{\color[named]{Black}{%
\special{pn 8}%
\special{pa 400 1280}%
\special{pa 2000 1280}%
\special{fp}%
}}%
% CIRCLE 2 0 0 0  
% 4 780 1280 830 1320 830 1320 830 1320
% 
{\color[named]{Black}{%
\special{pn 0}%
\special{sh 1.000}%
\special{ia 780 1280 30 30  0.0000000 6.2831853}%
}}%
{\color[named]{Black}{%
\special{pn 8}%
\special{ar 780 1280 30 30  0.0000000 6.2831853}%
}}%
% STR 2 0 3 0 Black White
% 4 560 1080 560 1180 2 0 0 0
% $D_4$
\put(4.0000,-12.3000){\makebox(0,0)[lb]{$\dfrac{1}{p}(1,1)$}}%
\put(4.0000,-17.0000){\makebox(0,0)[lb]{$x=0$}}
\put(-0.3000,-13.8000){\makebox(0,0)[lb]{$y=0$}}
\put(16.3000,-16.9800){\makebox(0,0)[lb]{$z=0$}}
\put(20.4000,-13.2000){\makebox(0,0)[lb]{$w=0$}}
\put(5.5000,-3.0000){\makebox(0,0)[lb]{$u=0$}}
\put(14.6000,-2.9300){\makebox(0,0)[lb]{$v=0$}}
\end{picture}%
}
\vspace{4mm}
\item Let $\delta = x \der/\der x + iy \der/\der y \in \Der_k\, k(\PP^2)$ with $i \in \FF_p, i \not= 0, 1$ and let 
$\pi: \PP^2 \rightarrow Y=\PP^2/\delta$ be the quotient map. For the local coordinates $z, w$ and $u, v$, we have 
$$\delta = -(z\der/\der z + (1-i)w\der/\der w) = -(iu\der/\der u + (i-1)v \der/\der v).$$
Hence the corresponding $1$-foliation $\mathcal{O}_{\PP^2} (\mathrm{div}(\delta))$ is isomorphic to $\mathcal{O}_{\PP^2}$, and $\pi(U_0)$, 
$\pi(U_1)$ and $\pi(U_0)$ have a toric singularity of type $\frac{1}{p}(1,i)$, $\frac{1}{p}(1,1-i)$ and $\frac{1}{p}(i,i-1)$, 
respectively. The configuration of the singular points of $Y$ is as follows:\vspace{4mm}\\
\center{
%WinTpicVersion4.08
\unitlength 0.1in
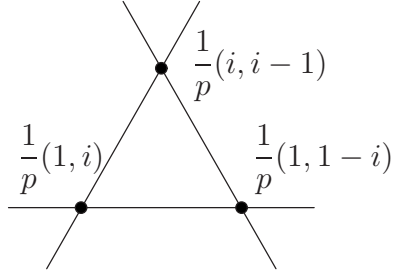
\begin{picture}( 16.0000, 14.0000)(  4.0000,-16.0000)
% LINE 2 0 3 0 Black White
% 2 600 1600 1400 200
% 
{\color[named]{Black}{%
\special{pn 8}%
\special{pa 600 1600}%
\special{pa 1400 200}%
\special{fp}%
}}%
% LINE 2 0 3 0 Black White
% 2 1000 200 1800 1600
% 
{\color[named]{Black}{%
\special{pn 8}%
\special{pa 1000 200}%
\special{pa 1800 1600}%
\special{fp}%
}}%
% LINE 2 0 3 0 Black White
% 2 400 1280 2000 1280
% 
{\color[named]{Black}{%
\special{pn 8}%
\special{pa 400 1280}%
\special{pa 2000 1280}%
\special{fp}%
}}%
% CIRCLE 2 0 0 0  
% 4 780 1280 830 1320 830 1320 830 1320
% 
{\color[named]{Black}{%
\special{pn 0}%
\special{sh 1.000}%
\special{ia 780 1280 30 30  0.0000000 6.2831853}%
}}%
{\color[named]{Black}{%
\special{pn 8}%
\special{ar 780 1280 30 30  0.0000000 6.2831853}%
}}%
% CIRCLE 2 0 0 0  
% 4 1620 1290 1670 1330 1670 1330 1670 1330
% 
{\color[named]{Black}{%
\special{pn 0}%
\special{sh 1.000}%
\special{ia 1620 1280 30 30  0.0000000 6.2831853}%
}}%
{\color[named]{Black}{%
\special{pn 8}%
\special{ar 1620 1280 30 30  0.0000000 6.2831853}%
}}%
% CIRCLE 2 0 0 0  
% 4 1180 550 1230 590 1230 590 1230 590
% 
{\color[named]{Black}{%
\special{pn 0}%
\special{sh 1.000}%
\special{ia 1200 550 30 30  0.0000000 6.2831853}%
}}%
{\color[named]{Black}{%
\special{pn 8}%
\special{ar 1200 550 30 30  0.0000000 6.2831853}%
}}%
% STR 2 0 3 0 Black White
% 4 560 1080 560 1180 2 0 0 0
% $D_4$
\put(4.5000,-12.1000){\makebox(0,0)[lb]{$\displaystyle\frac{1}{p} (1,i)$}}%
% STR 2 0 3 0 Black White
% 4 1360 470 1360 570 2 0 0 0
% $A_1$
\put(13.6000,-7.2000){\makebox(0,0)[lb]{$\displaystyle\frac{1}{p} (i,i-1)$}}%
% STR 2 0 3 0 Black White
% 4 1860 1400 1860 1500 2 0 0 0
% $A_1$
\put(16.8000,-12.1000){\makebox(0,0)[lb]{$\displaystyle\frac{1}{p} (1,1-i)$}}%
\end{picture}%
}
\vspace{4mm}
\item Suppose $p=2$. Let $\delta = x^2 \der/\der x + y^2 \der/\der y \in \Der_k\, k(\PP^2)$ and 
$\pi: \PP^2 \rightarrow Y=\PP^2/\delta$ the quotient map. For the local coordinates $z, w$ and $u, v$, we have 
$$\delta = 1/z(z\der/\der z + w(w+1)\der/\der w) = 1/u(u\der/\der u + v(v+1) \der/\der v).$$
Hence the corresponding $1$-foliation $\mathcal{O}_{\PP^2} (\mathrm{div}(\delta))$ is isomorphic to $\mathcal{O}_{\PP^2} (-1)$, and 
$\pi(U_1)$ (resp. $\pi(U_2)$) has two $A_1$-singularities at the points corresponding to $(0,0), (0,1) \in U_1$ (resp. 
$(0,0), (0,1) \in U_2$). Since 
$$\mathcal{O}_{\pi(U_0)} = \mathcal{O}_{U_0}^{\delta} = k[x^2,y^2,x^2y+xy^2] \cong k[X,Y,Z]/(Z^2+X^2Y+XY^2),$$
we see that $\pi(U_0)$ has a $D_4^0$-singularity (see \cite{A} for rational double points in positive characteristic). The 
configuration of the singular points of $Y$ is as follows:\vspace{4mm}\\
\center{
%WinTpicVersion4.08
\unitlength 0.1in
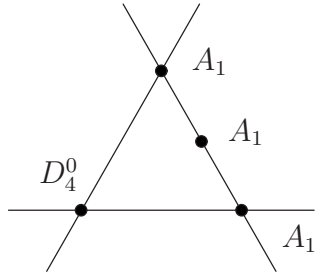
\begin{picture}( 16.0000, 14.0000)(  4.0000,-16.0000)
% LINE 2 0 3 0 Black White
% 2 600 1600 1400 200
% 
{\color[named]{Black}{%
\special{pn 8}%
\special{pa 600 1600}%
\special{pa 1400 200}%
\special{fp}%
}}%
% LINE 2 0 3 0 Black White
% 2 1000 200 1800 1600
% 
{\color[named]{Black}{%
\special{pn 8}%
\special{pa 1000 200}%
\special{pa 1800 1600}%
\special{fp}%
}}%
% LINE 2 0 3 0 Black White
% 2 400 1280 2000 1280
% 
{\color[named]{Black}{%
\special{pn 8}%
\special{pa 400 1280}%
\special{pa 2000 1280}%
\special{fp}%
}}%
% CIRCLE 2 0 0 0  
% 4 780 1280 830 1320 830 1320 830 1320
% 
{\color[named]{Black}{%
\special{pn 0}%
\special{sh 1.000}%
\special{ia 780 1280 30 30  0.0000000 6.2831853}%
}}%
{\color[named]{Black}{%
\special{pn 8}%
\special{ar 780 1280 30 30  0.0000000 6.2831853}%
}}%
% CIRCLE 2 0 0 0  
% 4 1620 1290 1670 1330 1670 1330 1670 1330
% 
{\color[named]{Black}{%
\special{pn 0}%
\special{sh 1.000}%
\special{ia 1620 1280 30 30  0.0000000 6.2831853}%
}}%
{\color[named]{Black}{%
\special{pn 8}%
\special{ar 1620 1280 30 30  0.0000000 6.2831853}%
}}%
% CIRCLE 2 0 0 0  
% 4 1180 550 1230 590 1230 590 1230 590
% 
{\color[named]{Black}{%
\special{pn 0}%
\special{sh 1.000}%
\special{ia 1200 550 30 30  0.0000000 6.2831853}%
}}%
{\color[named]{Black}{%
\special{pn 8}%
\special{ar 1200 550 30 30  0.0000000 6.2831853}%
}}%
% CIRCLE 2 0 0 0  
% 4 1410 920 1460 960 1460 960 1460 960
% 
{\color[named]{Black}{%
\special{pn 0}%
\special{sh 1.000}%
\special{ia 1410 920 30 30  0.0000000 6.2831853}%
}}%
{\color[named]{Black}{%
\special{pn 8}%
\special{ar 1410 920 30 30  0.0000000 6.2831853}%
}}%
% STR 2 0 3 0 Black White
% 4 560 1080 560 1180 2 0 0 0
% $D_4$
\put(5.6000,-11.8000){\makebox(0,0)[lb]{$D_4^0$}}%
% STR 2 0 3 0 Black White
% 4 1360 470 1360 570 2 0 0 0
% $A_1$
\put(13.6000,-5.7000){\makebox(0,0)[lb]{$A_1$}}%
% STR 2 0 3 0 Black White
% 4 1590 880 1590 980 2 0 0 0
% $A_1$
\put(15.5000,-9.4000){\makebox(0,0)[lb]{$A_1$}}%
% STR 2 0 3 0 Black White
% 4 1860 1400 1860 1500 2 0 0 0
% $A_1$
\put(18.3000,-15.0000){\makebox(0,0)[lb]{$A_1$}}%
\end{picture}%
}
\vspace{4mm}
\item Suppose $p=2$. Let $\delta = xy^2 \der/\der x + (x^2+y^3) \der/\der y \in \Der_k\, k(\PP^2)$ and 
$\pi: \PP^2 \rightarrow Y=\PP^2/\delta$ the quotient map. For the local coordinates $z, w$ and $u, v$, we have 
$$\delta = 1/z(w^2\der/\der z + \der/\der w) = 1/u((1+uv^2)\der/\der u + v^3 \der/\der v).$$
Hence the corresponding $1$-foliation $\mathcal{O}_{\PP^2} (\mathrm{div}(\delta))$ is isomorphic to $\mathcal{O}_{\PP^2} (-1)$, and $Y$ 
is smooth on $\pi(U_1)$ and $\pi(U_2)$. Since 
$$\mathcal{O}_{\pi(U_0)} = \mathcal{O}_{U_0}^{\delta} = k[x^2,y^2,x^3+xy^3] \cong k[X,Y,Z]/(Z^2+X^3+XY^3),$$
we see that $\pi(U_0)$ has an $E_7^0$-singularity. The configuration of the singular points of $Y$ is as follows:\vspace{4mm}\\
\center{
%WinTpicVersion4.08
\unitlength 0.1in
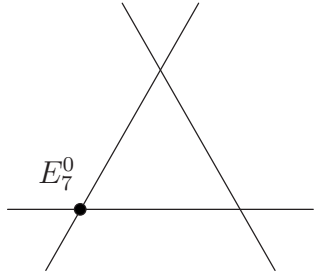
\begin{picture}( 16.0000, 14.0000)(  4.0000,-16.0000)
% LINE 2 0 3 0 Black White
% 2 600 1600 1400 200
% 
{\color[named]{Black}{%
\special{pn 8}%
\special{pa 600 1600}%
\special{pa 1400 200}%
\special{fp}%
}}%
% LINE 2 0 3 0 Black White
% 2 1000 200 1800 1600
% 
{\color[named]{Black}{%
\special{pn 8}%
\special{pa 1000 200}%
\special{pa 1800 1600}%
\special{fp}%
}}%
% LINE 2 0 3 0 Black White
% 2 400 1280 2000 1280
% 
{\color[named]{Black}{%
\special{pn 8}%
\special{pa 400 1280}%
\special{pa 2000 1280}%
\special{fp}%
}}%
% CIRCLE 2 0 0 0  
% 4 780 1280 830 1320 830 1320 830 1320
% 
{\color[named]{Black}{%
\special{pn 0}%
\special{sh 1.000}%
\special{ia 780 1280 30 30  0.0000000 6.2831853}%
}}%
{\color[named]{Black}{%
\special{pn 8}%
\special{ar 780 1280 30 30  0.0000000 6.2831853}%
}}%
% STR 2 0 3 0 Black White
% 4 560 1080 560 1180 2 0 0 0
% $D_4$
\put(5.6000,-11.8000){\makebox(0,0)[lb]{$E_7^0$}}%
\end{picture}%
}
\end{enumerate}
\end{exa}

If $X$ is a globally F-regular variety, then $\mathcal{O}_{X,x}$ is strongly F-regular for all $x \in X$ (see e.g., \cite{Smith} for details). 
Now $D_4^0$ and $E_7^0$-singularities are not strongly F-regular, so that $F$-sandwiches in the above example (3) and 
(4) are not globally F-regular. On the other hand, $F$-sandwiches $Y$ in (1) and (2) are globally F-regular, since there 
exists a (global) splitting $1-\delta^{p-1} : \pi_{\ast}\mathcal{O}_{\PP^2} \rightarrow \mathcal{O}_Y$. Moreover we see that globally 
F-regular $F$-sandwiches of $\PP^2$ of degree $p$ are only those. 

Let $N = \ZZ^2$ be a lattice with standard basis $e_1=(1,0), e_2=(0,1)$. For a fan $\Sigma$ in $N\otimes \R$, we denote 
the associated toric variety over $k$ by $X_{\Sigma}$. (For the general theory of toric varieties, we refer to \cite{CLS}.)
For each $1 \leq i \leq p-1$, let ${\Sigma}_i$ be the complete fan whose rays are spanned by $e_2$, $pe_1-ie_2$ and 
$-pe_1 + (i-1)e_2$. 

\begin{thm}[\cite{HS}]\label{f-sand-p2}
A globally F-regular $F$-sandwich of $\PP^2$ of degree $p$ is isomorphic to either one of the singular toric surfaces 
$X_{{\Sigma}_i}$ $(1 \leq i \leq p-1)$. In particular, there are just $p-1$ isomorphism classes of globally F-regular 
F-sandwiches of $\PP^2$ of degree $p$. 
\end{thm}

In \cite{HS}, we have seen that globally F-regular $F$-sandwiches of $\PP^2$ have at most three singular points and 
given the classification by changing coordinates so that those singular points are located at the origins of three standard 
affine patches. A classification of globally F-regular $F$-sandwiches of Hirzebruch surfaces will be given by a similar 
argument, that is, we will give the classification by changing coordinates so that their singular points are located at the 
origins of four standard affine patches. 

The Hirzebruch surface $\mathcal{H}_d$ of index $d \geq 0$ is the $\PP^1$-bundle associated to the vector bundle 
$\mathcal{O}_{{\PP}^1} \oplus \mathcal{O}_{{\PP}^1} (d)$ on $\PP^1$. It is well-known that $\mathcal{H}_d$ is the union of affine open sets 
$U_i$ $(i=1, 2, 3, 4)$ whose affine coordinate rings are $k[x,y]$, $k[z,w]=k[x^d y,1/x]$, $k[s,t]=k[1/x,1/x^d y]$ and 
$k[u,v]=k[1/y, x]$, respectively. $\mathcal{H}_d$ is a toric surface given by the complete fan whose rays are spanned by 
$\rho_1 = e_2$, $\rho_2 = e_1$, $\rho_3 = -e_2$ and $\rho_4 = -e_1 + d e_2$. For each $i=1,2,3,4$, let $D_i$ be the toric 
divisor that corresponds to the ray spanned by $\rho_i$.

Let ${\Sigma}_{di}$ ($1 \leq i \leq p-1$) (resp. ${\Sigma}_{d0}$ ; ${\Sigma}_{dp}$) be the complete fan whose rays are 
spanned by $e_2$, $p e_1-i e_2$, $-e_2$ and $-p e_1 + (i+d)e_2$ (resp. $e_2$, $e_1$, $-e_2$ and $-p e_1 + d e_2$ ; $e_2$, 
$e_1$, $-e_2$ and $-e_1 + d p e_2$).

First we consider the case where $d=0$. Let $\delta \in \Der_k (k(\mathcal{H}_0))$. Suppose that the corresponding 
$1$-foliation $\mathcal{O}_{\mathcal{H}_0} (\mathrm{div}(\delta))$ has a nonzero global section. Then we easily see that 
$\delta \sim (a_2x^2 + a_1x + a_0) {\partial}/{\partial x} + (b_2y^2 + b_1y +b_0) {\partial}/{\partial y}$, where $a_i, b_j \in k$. 
Let $Y$ be a globally F-regular $F$-sandwich of $\mathcal{H}_0$ of degree $p$ with the finite morphism 
$\pi : \mathcal{H}_0 \rightarrow Y$ through which the Frobenius morphism of $\mathcal{H}_0$ factors, and $L$ be the 
corresponding $1$-foliation. Then we have $\hom_{\mathcal{O}_Y} (\pi_{\ast} \mathcal{O}_{\mathcal{H}_0}, \mathcal{O}_Y) \cong \pi_{\ast} L^{\otimes (p-1)}$ 
by the adjunction formula and the canonical bundle formula (see the proof of Theorem 2.4). Since a splitting of the 
associated ring homomorphism $\mathcal{O}_Y \rightarrow \pi_{\ast} \mathcal{O}_{\mathcal{H}_0}$ is its nonzero global section, we see that 
$L$ has a nonzero global section. Thus we may assume that the corresponding $p$-closed rational vector field $\delta$ 
is equivalent to $(a_2x^2 + a_1x + a_0) {\partial}/{\partial x} + (b_2y^2 + b_1y +b_0) {\partial}/{\partial y}$, where $a_i, b_j \in k$. 
By changing coordinates so that their singular points are located at the origins of $U_i$, we see that a globally F-regular 
$F$-sandwich of $\mathcal{H}_0$ is isomorphic to either one of the toric surfaces $X_{{\Sigma}_{0i}}$ $(0 \leq i \leq p-1)$. 
In this case, there are just $p$ isomorphism classes of the globally F-regular $F$-sandwiches. 

In the case where $d \geq 1$, the situation is slightly complicated.

\begin{lmm}\label{form_of_del-hir}
Suppose that $d \geq 1$. Let $\delta \in \Der_k (k(\mathcal{H}_d))$. If the corresponding $1$-foliation 
$\mathcal{O}_{\mathcal{H}_d} (\mathrm{div}(\delta))$ has a nonzero global section, then 
$$\delta \sim (a_2x^2 + a_1x + a_0) {\partial}/{\partial x} + (F(x)y-d a_2 x + b)y {\partial}/{\partial y},$$
where $a_i, b \in k$ and $F(x) \in k[x]$ with $\deg F(x) \leq d$. 
\end{lmm}

\begin{proof}
Let $\delta = {f \partial}/{\partial x} + {g \partial}/{\partial y}$. If $f=0$ or $g=0$, then 
$\delta \sim {\partial}/{\partial y}$ or ${\partial}/{\partial x}$, so there is nothing to prove. Suppose that $f\not=0$ and 
$g\not=0$. Multiplying by a suitable rational function on $\mathcal{H}_d$, we have 
$$\delta \sim \sum_{0 \leq i,j} a_{ij}x^i y^j \frac{\partial}{\partial x} + \sum_{0 \leq m,n} b_{mn}x^my^n \frac{\partial}{\partial y},$$
where $a_{ij}, b_{mn} \in k$, and $\sum_{0 \leq i,j} a_{ij}x^i y^j$ and $\sum_{0 \leq m,n} b_{mn}x^my^n$ have no common 
factors. If we express $\delta$ for the local coordinates $u, v$, then we see that the coefficient of $\delta$ in 
${\partial}/{\partial u}$ (resp. ${\partial}/{\partial v}$) equals $-\sum_{0 \leq m,n} {b_{mn}v^m}/{u^{n-2}}$ (resp. 
$\sum_{0 \leq i,j} {a_{ij}v^i}/{u^j}$). If $b_{mn}\not=0$ for some $n \geq 3$ or $a_{ij} \not= 0$ for some $j \geq 1$, then 
$\ord_{D_3} \mathrm{div}(\delta) < 0$. Since $\sum_{0 \leq i,j} a_{ij}x^i y^j$ and $\sum_{0 \leq m,n} b_{mn}x^my^n$ have 
no common factors, this means that $H^0 ( \mathcal{H}_d ,\mathcal{O}_{\mathcal{H}_d} (\mathrm{div}(\delta)))=0$, which 
is a contradiction. Therefore we have $b_{mn}=0$ for $n \geq 3$ and $a_{ij} = 0$ for $j \geq 1$. Thus we have 
$\delta \sim \sum_{0 \leq i} a_{i0}x^i {\partial}/{\partial x} + \sum_{0 \leq m, 0 \leq n \leq 2} b_{mn}x^my^n {\partial}/{\partial y}$. 
We denote the right-hand side by ${\delta}'$. 

If we express ${\delta}'$ for the local coordinates $s, t$, then we see that the coefficient of ${\delta}'$ in 
${\partial}/{\partial s}$ equals $-\sum_{0 \leq i} {a_{i0}}/{s^{i-2}}$. If $a_{i0} \not= 0$ for some $i \geq 3$, then 
$\ord_{D_3} \mathrm{div}(\delta) \leq 0$ and $\ord_{D_4} \mathrm{div}(\delta) < 0$. This means that 
$H^0 ( \mathcal{H}_d ,\mathcal{O}_{\mathcal{H}_d} (\mathrm{div}(\delta)))=0$, which is a contradiction.Therefore we 
have $a_{i0} = 0$ for $i \geq 3$. Thus we have ${\delta}' = (a_{20} x^2 + a_{10} x + a_{00}) {\partial}/{\partial x} 
+ (F_2(x)y^2 +F_1(x)y +F_0(x) ) {\partial}/{\partial y}$, where $F_i \in k[x]$. Considering ${\delta}'$ for the local 
coordinates $s, t$ again, we see that $\deg F_2 (x) \leq d$, $F_1 (x) = -d a_{20}x$ and $F_0(x)=0$. This means 
that $\delta \sim (a_2 x^2 + a_1x + a_0) {\partial}/{\partial x} + (F(x)y-d a_2x + b)y {\partial}/{\partial y}$, where 
$a_i, b \in k$ and $F(x) \in k[x]$ with $\deg F(x) \leq d$. 
\end{proof}

\begin{thm}\label{f-sand-hir}
A globally $F$-regular $F$-sandwich of the Hirzebruch surface $\mathcal{H}_d$ of index $d \geq 1$ of degree $p$ 
is isomorphic to either one of the toric surfaces $X_{\Sigma_{di}}$ $(0 \leq i \leq p)$. In particular, there are just 
$p+1$ isomorphism classes of globally $F$-regular $F$-sandwiches of $\mathcal{H}_d$ of degree $p$. 
\end{thm}

\begin{proof}
Let $Y$ be a globally F-regular $F$-sandwich of $\mathcal{H}_d$ of degree $p$ with the finite morphism 
$\pi : \mathcal{H}_d \rightarrow Y$ through which the Frobenius morphism of $\mathcal{H}_d$ factors, and 
let $L \subset T_{\mathcal{H}_d}$ (resp. $\delta \in \Der_k k(\mathcal{H}_d)$) be the corresponding 
$1$-foliation (resp. the $p$-closed rational vector field). Since the associated ring homomorphism 
$\mathcal{O}_Y \to \pi_*\mathcal{O}_{\mathcal{H}_d}$ splits by Lemma~\ref{split_vs_F-reg}, there is a nonzero $\mathcal{O}_Y$-module 
homomorphism ${\pi}_{\ast}\mathcal{O}_{\mathcal{H}_d} \rightarrow  \mathcal{O}_Y$. Outside $\Sing\, L$ we have 
\begin{eqnarray*}
 \hom_{\mathcal{O}_Y}(\pi_*\mathcal{O}_{\mathcal{H}_d}, \mathcal{O}_Y)
 &\cong & \hom_{\mathcal{O}_Y} (\pi_*\mathcal{O}_{\mathcal{H}_d},\omega_Y) \otimes \omega_Y^{-1} \\
 &\cong & \pi_*(\omega_{\mathcal{H}_d}) \otimes \omega_Y^{-1} \\
 &\cong & \pi_*(\omega_{\mathcal{H}_d} \otimes \pi^*(\omega_Y^{-1})) \\
 &\cong & \pi_*(L^{\otimes (p-1)}),
\end{eqnarray*}
which gives a (global) isomorphism 
$\Hom_{\mathcal{O}_Y}(\pi_*\mathcal{O}_{\mathcal{H}_d},\mathcal{O}_Y) \cong H^0 (\mathcal{H}_d,L^{\otimes (p-1)})$ 
since $Y$ is normal. Thus $L^{\otimes (p-1)}$ has a nonzero global section. Since $L$ is a line bundle on 
$\mathcal{H}_d$, $L$ has a nonzero global section.

If $Y$ is singular at a point on the image of the negative section $D_1$, after a suitable change of coordinates, 
we may assume that $Y$ is singular at the point corresponding to the origin of $U_1$. (In what follows we 
refer to this assumption as ``the assumption for the arrangement of singularities''.) 

Then by Lemma~\ref{form_of_del-hir} we may assume that 
$$\delta = (a_2x^2 + a_1x + a_0) {\partial}/{\partial x} + (F(x)y -d a_2 x + b)y {\partial}/{\partial y},$$
where $a_i, b \in k$ and $F(x) \in k[x]$ with $\deg F(x) \leq d$. We will proceed by changing coordinates so 
that singular points are located at the origins of $U_i$. 

We divide the remainder of the proof into three steps: 

(1) Suppose that $a_2=a_1=0$, i.e., $\delta = a_0 {\partial}/{\partial x} + (F(x)y + b)y {\partial}/{\partial y}$. 
If $a_0=0$ then $\delta \sim {\partial}/{\partial y}$. Suppose that $a_0 \not=0$. Then we may assume that 
$\delta = {\partial}/{\partial x} - (F(x)y + b)y {\partial}/{\partial y}$. If we express $\delta$ for the local 
coordinates $u, v$, we have $\delta = {\partial}/{\partial v} - (F(v) + bu) {\partial}/{\partial u}$. Since 
$\delta$ is $p$-closed, we see that $b=0$ and the $(p-1)$-th derivative $F^{(p-1)}(v)=0$. In particular, 
$\delta = {\partial}/{\partial x} - F(x)y^2 {\partial}/{\partial y}$. For the local coordinates $z, w$, we have 
$\delta = z(d w-z F(1/w)w^d){\partial}/{\partial z} - w^2 {\partial}/{\partial w}$. If $\deg F = d$, then the 
constant term of $F(1/w)w^d$ is not zero, so that $Y$ is singular at the point corresponding to the origin 
of $U_2$. On the other hand, $Y$ is smooth at the point corresponding to the origin of $U_1$, since 
$\delta = {\partial}/{\partial x} - (F(x)y + b)y {\partial}/{\partial y}$ for the local coordinates $x, y$. This 
contradicts the assumption for the arrangement of singularities. Thus $\deg F \leq d-1$. Then there exists 
a polynomial $G(v) \in k[v]$ such that $\deg G(v) \leq d$ and its derivative equals $F(v)$ since the 
$(p-1)$-th derivative $F^{(p-1)}(v)=0$. After a change of a coordinate $u + G(v) \mapsto u$, we have 
$\delta = {\partial}/{\partial v}$. For the local coordinates $x, y$, we have $ \delta = {\partial}/{\partial x}$. 

(2) Suppose that $a_2=0$ and $ a_1\not= 0$. Then we may assume that 
$\delta = (x + a_0){\partial}/{\partial x} + (F(x)y + b)y {\partial}/{\partial y}$. If $(x-a_0)|(F(x)y +b)$, then 
$(x-a_0)|F(x)$ and $b=0$. This means that $\delta \sim {\partial}/{\partial x} - G(x)y^2 {\partial}/{\partial y}$, 
where $G(x)=F(x)/(x-c)$, and this is the case (1). Now we assume that $(x-a_0)\nmid (F(x)y +b)$. From the 
assumption for the arrangement of singularities we see that $a_0=0$, i.e., 
$\delta = x{\partial}/{\partial x} + (F(x)y + b)y {\partial}/{\partial y}$. For the local coordinates $u, v$, we have 
$ \delta = v{\partial}/{\partial v} +(F(v)+d u){\partial}/{\partial u}$. Since $\delta$ is $p$-closed, we see that 
$$\delta = v{\partial}/{\partial v} +(\sum_{\substack{0\leq j \leq d,\\ j \not\equiv i ~\Mod p}} c_j v^j+i u){\partial}/{\partial u},$$
where $i \in ({\ZZ}/{p\ZZ})^{\times}$, $c_j \in k$. After a change of a coordinate  
$$u -\hspace{-1.5mm} \sum_{\substack{0\leq j \leq d,\\ j \not\equiv i ~\Mod p}} {c_j v^j}/{(j-i)}\hspace{2.5mm} \mapsto \hspace{2.5mm} u,$$
we have $\delta = v {\partial}/{\partial v} + i u {\partial}/{\partial u}$. For the local coordinates $x, y$, we have 
$\delta = x{\partial}/{\partial x} -i y{\partial}/{\partial y}$.

(3) Suppose that $a_2 \not= 0$. Then we may assume that 
$\delta = (x^2 +a_1x + a_0){\partial}/{\partial x} + (F(x)y -d x + b)y {\partial}/{\partial y}$. 
If $(x-A)|(x^2 + a_1x+a_0)$ and $(x-A)|(F(x)y -d x + b)$ for some $A \in k$, then we have 
$\delta \sim (x-B) {\partial}/{\partial x} + (G(x)y -b)y {\partial}/{\partial y}$, where $B \in k$ is the other root 
of $x^2 + a_1x + a_0 =0$ and $G(x)=F(x)/(x-A) \in k[x]$. This is the case (2). We assume that $x^2 + a_1x+a_0$ 
and $F(x)y -d x + b$ have no common factor. Then from the assumption for the arrangement of singularities 
we see that $a_0=0$. If $a_1=0$, then $\delta = x^2{\partial}/{\partial x} - (F(x)y -d x + b)y {\partial}/{\partial y}$. 
If $x \nmid (F(x)y -d x + b)$, then $\delta$ must be nilpotent, which is a contradiction. Therefore $x|(F(x)y -d x + b)$. 
Then $\delta \sim x{\partial}/{\partial x} - (G(x)y-b)y {\partial}/{\partial y}$, where $G(x)=F(x)/x \in k[x]$. 
This is the case (2). Now we assume that $a_1 \not= 0$. If we express $\delta$ for the local coordinates $z, w$, 
we have $\delta = (F(1/w)w^dz^2+(d a_1+b)z){\partial}/{\partial z} -(1+a_1w){\partial}/{\partial w}$, where 
$G(w)=F(1/w)w^d \in k[w]$. After a change of a coordinate $w+{a_1}^{-1} \mapsto w$, we have 
$\delta = (G(w-{a_1}^{-1})z^2 +(d{a_1} +b)z) {\partial}/{\partial z} -{a_1}^{-1} w{\partial}/{\partial w}$. For the local 
coordinates $x, y$, we have $ \delta \sim {a_1}^{-1} x{\partial}/{\partial x} + (H(x)y +b)y {\partial}/{\partial y}$, 
where $H(x) \in k[x]$. This is the case (2). 

Therefore we see that $\delta \sim x{\partial}/{\partial x} + i y {\partial}/{\partial y}$ $(i \in {\ZZ}/{p\ZZ})$ 
or ${\partial}/{\partial y}$. Then we can easily check that $\mathcal{H}_d / \delta$ is isomorphic to either one of 
the toric surfaces $X_{\Sigma_{di}}$ $(0 \leq i \leq p)$. 
\end{proof}

\begin{rem}
Ganong and Russell showed in \cite{GR} that for each Hirzebruch surface there are at most two smooth 
$F$-sandwiches. The smooth $F$-sandwich is given by the quotient of $\mathcal{H}_d$ by the rational vector 
field ${\partial}/{\partial x}$ or ${\partial}/{\partial y}$. 
\end{rem}

%\bibliographystyle{amsplain}
%\bibliography{ref}

\end{document}